\documentclass[10pt]{article}

\usepackage{amsmath,amssymb,amsthm}

\renewenvironment{proof}{\noindent {\bfseries Proof.}}{\hfill $\Box$}

\usepackage{a4wide}

\usepackage{cite}

\makeatletter
\def\@biblabel#1{#1.}
\makeatother

\newtheorem{theorem}{Theorem}[section]

\newtheorem{lemma}{Lemma}[section]
\newtheorem{definition}{Definition}[section]
\newtheorem{remark}{Remark}[section]

\begin{document}

\title{The DuBois--Reymond Fundamental Lemma of the Fractional Calculus of Variations
and an Euler--Lagrange Equation Involving only Derivatives of Caputo}

\author{Matheus J. Lazo$^{1}$
\and
Delfim F. M. Torres$^{2}$}

\date{}


{\renewcommand{\thefootnote}{}

\footnotetext{{\bf This is a preprint of a paper whose final and definite
form will appear in \emph{Journal of Optimization Theory and Applications}
(JOTA). Paper submitted 26-July-2012; revised 12-Sept-2012 and 22-Sept-2012;
accepted for publication 02-Oct-2012.}}}


\footnotetext[1]{Instituto de Matem\'{a}tica, Estat\'{\i}stica e F\'{\i}sica
--- FURG, Rio Grande, RS, Brazil.
E-mail: matheuslazo@furg.br.}

\footnotetext[2]{Corresponding author.
Center for Research and Development in Mathematics and Applications,
Department of Mathematics, University of Aveiro, 3810-193 Aveiro, Portugal.
E-mail: delfim@ua.pt.}


\maketitle


\begin{abstract}
Derivatives and integrals of non-integer order
were introduced more than three centuries ago,
but only recently gained more attention due to their
application on nonlocal phenomena. In this context,
the Caputo derivatives are the most popular approach
to fractional calculus among physicists, since differential
equations involving Caputo derivatives require
regular boundary conditions. Motivated by several applications
in physics and other sciences, the fractional calculus
of variations is currently in fast development. However,
all current formulations for the fractional variational calculus fail
to give an Euler--Lagrange equation with only Caputo derivatives.
In this work, we propose a new approach to the fractional calculus
of variations by generalizing the DuBois--Reymond lemma
and showing how Euler--Lagrange equations
involving only Caputo derivatives can be obtained.

\bigskip


\bigskip


\end{abstract}


\section{Introduction}

The fractional calculus with derivatives and integrals of non-integer order
started more than three centuries ago, with l'H\^opital and Leibniz,
when the derivative of order $1/2$ was suggested \cite{OldhamSpanier}.
This subject was then considered by several mathematicians like Euler, Fourier,
Liouville, Grunwald, Letnikov, Riemann, and many others up to nowadays.
Although the fractional calculus is almost as old as the usual integer order calculus,
only in the last three decades it has gained more attention due to its many applications
in various fields of science, engineering, economics, biomechanics, etc.
(see \cite{SATM,Kilbas,Hilfer,Magin,HerrmannBook} for a review).

Fractional derivatives are nonlocal operators and are historically applied
in the study of nonlocal or time dependent processes. The first and well
established application of fractional calculus in physics was in the framework
of anomalous diffusion, which is related to features observed in many physical systems,
e.g., in dispersive transport in amorphous semiconductor, liquid crystals,
polymers, proteins, etc. \cite{Metzler2,Klages}. Recently,
the study of nonlocal quantum phenomena through fractional calculus began a fast development,
where the nonlocal effects are due to either long-range interactions
or time-dependent processes with many scales
\cite{Hilfer,Laskin2,Naber,Iomin,Tarasov1}.
Relativistic quantum mechanics \cite{KP,Zavada,MAB},
field theories \cite{Tarasov2,Herrmann,Lazo},
and gravitation \cite{Munkhammar}, have been considered
in the context of fractional calculus. The physical and geometrical meaning
of the fractional derivatives has been investigated by several authors \cite{NM,Podlubny}.
It has been shown that the Stieltjes integral
can be interpreted as the real distance passed by a moving object \cite{NM,Podlubny}.
The most popular operators of fractional calculus,
namely the Riemann--Liouville \cite{OldhamSpanier} and Caputo \cite{caputo,caputo2}
fractional operators, are then given similar physical interpretations \cite{NM,Podlubny}.

One of the most remarkable applications of fractional calculus
in physics is in the context of classical mechanics. Riewe
showed that a Lagrangian involving fractional time derivatives
leads to an equation of motion with nonconservative forces, such
as friction \cite{Riewe,Riewe2}. It is a striking result
since frictional and nonconservative forces
are beyond the usual macroscopic variational treatment and, consequently,
beyond the most advanced methods of classical mechanics \cite{Bauer}.
Riewe generalized the usual variational calculus to Lagrangians
depending on fractional derivatives, in order to deal
with nonconservative forces \cite{Riewe,Riewe2}. Recently, several approaches
have been developed in order to generalize the least action principle
and the Euler--Lagrange equations to include fractional derivatives
\cite{Riewe,Riewe2,Agrawal,BA,Cresson,APT,AT,OMT,MyID:226,MyID:227}.
In this new formalism, because of the nonlocal properties of fractional time derivatives,
the Euler--Lagrange equations appear to not respect
the causality principle. This difficulty is under investigation
by several authors \cite{CI,DY}. On the other hand,
in all the above  mentioned references,
the Euler--Lagrange equations always depend on the Riemann--Liouville
or mixed Caputo and Riemann--Liouville derivatives.
For the state of the art of the fractional calculus of variations
and respective fractional Euler--Lagrange equations, we refer the reader
to the recent book \cite{MyID:208}.

It is important to remark that while the Riemann--Liouville fractional derivatives \cite{OldhamSpanier}
are histo\-ri\-cally the most studied approach to fractional calculus, the Caputo \cite{caputo,caputo2} approach
to fractional derivatives is the most popular among physicists and scientists,
because the differential equations defined in terms of Caputo derivatives require regular
initial and boundary conditions. Furthermore, differential equations with Riemann--Liouville
derivatives require non-standard fractional initial and boundary conditions that lead,
in general, to singular solutions, thus limiting their application in physics and science \cite{HerrmannBook}.
On the other hand, within current formulations of the fractional calculus of variations,
even Lagrangians depending only on Caputo derivatives lead to Euler--Lagrange equations
with Riemann--Liouville derivatives (see \cite{AMT} and references therein).
This is a consequence of the Lagrange method to extremize functionals:
application of integration by parts for Caputo derivatives
in the Gateaux derivative of the functional, relates Caputo
to Riemann--Liouville derivatives. In the present work, we propose a different
approach, by generali\-zing the DuBois--Reymond fundamental lemma
to variational functionals with Caputo derivatives.
This new result enable us to obtain an Euler--Lagrange equation in integral form,
containing only Caputo derivatives. Moreover,
when the Lagrangian is $C^2$ in its domain,
we obtain an Euler--Lagrange fractional differential
equation depending only on Caputo derivatives
by taking the Caputo derivative on both sides
of the Euler--Lagrange equation in integral form.

The article is organized as follows. In Section~\ref{sec:2}, we review the basic notions
of Riemann--Liouville and Caputo fractional calculus, that are needed to formulate
the fractional problem of the calculus of variations. The fractional DuBois--Reymond
lemma is proved in Section~\ref{sec:3}, while in Section~\ref{sec:4} we obtain
the Euler--Lagrange equations in integral and differential forms. Finally, Section~\ref{sec:5} presents
the main conclusions of our work. Although we adopt here the Caputo fractional calculus,
the results can be straightforward rewritten within other approaches,
like the Riemann--Liouville one.


\section{The Fractional Calculus}
\label{sec:2}

There are several definitions of fractional order derivatives.
These definitions include the Riemann--Liouville, Caputo, Riesz,
Weyl, and Grunwald--Letnikov operators
\cite{OldhamSpanier,SATM,Kilbas,Hilfer,Magin,SKM,Diethelm}.
In this Section, we review some definitions and properties
of the Caputo and Riemann--Liouville fractional calculus.
Although our first main objective is to obtain a DuBois--Reymond
lemma for functionals with Caputo fractional derivatives,
the Riemann--Liouville derivatives are also defined here
because they naturally arise in the formulation.

Despite many different approaches to fractional calculus,
several known formulations are connected with the analytical
continuation of the Cauchy formula for $n$-fold integration.

\begin{theorem}[Cauchy formula]
Let $f:[a,b]\rightarrow \mathbb{R}$ be Riemann integrable in $[a,b]$.
The $n$-fold integration of $f$, $n\in \mathbb{N}$, is given by
\begin{equation}
\label{a2}
\begin{split}
\int_{a}^x f(\tilde{x})(d\tilde{x})^{n}
&= \int_{a}^x\int_{a}^{x_{n}}\int_{a}^{x_{n-1}}
\cdots \int_{a}^{x_3}\int_{a}^{x_2} f(x_1)dx_1dx_2\cdots dx_{n-1}dx_{n}\\
&= \frac{1}{\Gamma(n)}\int_{a}^x \frac{f(u)}{(x-u)^{1-n}}du ,
\end{split}
\end{equation}
where $\Gamma$ is the Euler gamma function.
\end{theorem}

The proof of Cauchy's formula \eqref{a2} can be found in several textbooks,
for example, it can be found in \cite{OldhamSpanier}.
The analytical continuation of \eqref{a2} gives a definition
for integration of non integer (or fractional) order.
This fractional order integration is the building block
of the Riemann--Liouville and Caputo calculus, the two most popular
formulations of fractional calculus, as well as several other approaches
\cite{OldhamSpanier,SATM,Kilbas,Hilfer,Magin,SKM,Diethelm}.
The fractional integrations obtained from \eqref{a2}
are historically called Riemann--Liouville fractional integrals.

\begin{definition}[Left and right Riemann--Liouville fractional integrals]
Let $\alpha \in \mathbb{R}_+$.
The operators $_aJ^{\alpha}_x$ and $_xJ^{\alpha}_b$, defined on $L_1([a,b])$ by
\begin{equation}
\label{a3}
_aJ^{\alpha}_x f(x)
:=\frac{1}{\Gamma(\alpha)}\int_{a}^x \frac{f(u)}{(x-u)^{1-\alpha}}du
\end{equation}
and
\begin{equation}
\label{a4}
_xJ^{\alpha}_b f(x)
:=\frac{1}{\Gamma(\alpha)}\int_x^b \frac{f(u)}{(u-x)^{1-\alpha}}du,
\end{equation}
where $a,b\in \mathbb{R}$ with $a<b$, are called the left and the right
Riemann--Liouville fractional integrals of order $\alpha$, respectively.
\end{definition}

For an integer $\alpha$, the fractional Riemann--Liouville integrals
\eqref{a3} and \eqref{a4} coincide with the usual integer order
$n$-fold integration \eqref{a2}. Moreover, from definitions
\eqref{a3} and \eqref{a4}, it is easy to see that the Riemann--Liouville
fractional integrals converge for any integrable function $f$ if $\alpha>1$.
Furthermore, it is possible to prove the convergence of \eqref{a3} and \eqref{a4}
even when $0<\alpha<1$ \cite{Diethelm}.

The integration operators $_aJ^{\alpha}_x$ and $_xJ^{\alpha}_b$
play a fundamental role in the definition of Caputo and Riemann--Liouville fractional derivatives.
In order to define the Riemann--Liouville derivatives, we recall that, for positive integers $n>m$,
it follows the identity $D^m_x f(x)=D^{n}_x {_aJ}^{n-m}_x f(x)$,
where $D^m_x$ is the ordinary derivative $d^m/dx^m$ of order $m$.

\begin{definition}[Left and right Riemann--Liouville fractional derivatives]
The left and the right Riemann--Liouville fractional derivatives
of order $\alpha \in \mathbb{R}_+$ are defined,
respectively, by
$$
_aD^{\alpha}_x f(x) := D^{n}_x {_aJ}^{n-\alpha}_x f(x)
$$
and
$$
_xD^{\alpha}_b f(x) := (-1)^nD^{n}_x {_xJ}^{n-\alpha}_b f(x)
$$
with $n=[\alpha]+1$; that is,
\begin{equation}
\label{a5}
_aD^{\alpha}_x f(x)
:= \frac{1}{\Gamma(n-\alpha)}\frac{d^n}{dx^n}
\int_{a}^x \frac{f(u)}{(x-u)^{1+\alpha-n}}du
\end{equation}
and
\begin{equation}
\label{a6}
_xD^{\alpha}_b f(x)
:= \frac{(-1)^n}{\Gamma(n-\alpha)}\frac{d^n}{dx^n}
\int_{x}^b \frac{f(u)}{(u-x)^{1+\alpha-n}}du.
\end{equation}
\end{definition}

The Caputo fractional derivatives are defined
in a similar way as the Riemann--Liouville ones,
but exchanging the order between integration and differentiation.

\begin{definition}[Left and right Caputo fractional derivatives]
The left and the right Caputo fractional derivatives of order
$\alpha\in \mathbb{R}_+$ are defined, respectively,
by ${_a^C D}^{\alpha}_x f(x) := {_aJ}^{n-\alpha}_x D^{n}_x f(x)$
and ${_x^C}D^{\alpha}_b f(x) := (-1)^n _xJ^{n-\alpha}_b
D^{n}_x f(x)$ with $n=[\alpha]+1$; that is,
\begin{equation}
\label{a7}
{_a^C D}^{\alpha}_x f(x) := \frac{1}{\Gamma(n-\alpha)}
\int_{a}^x \frac{f^{(n)}(u)}{(x-u)^{1+\alpha-n}}du
\end{equation}
and
\begin{equation}
\label{a8}
{_x^C D}^{\alpha}_b f(x)
:= \frac{(-1)^n}{\Gamma(n-\alpha)}\int_{x}^b
\frac{f^{(n)}(u)}{(u-x)^{1+\alpha-n}}du,
\end{equation}
where $a \le x \le b$ and $f^{(n)}(u)=\frac{d^n f(u)}{du^n} \in L_1([a,b])$
is the ordinary derivative of integer order $n$.
\end{definition}

An important consequence of definitions \eqref{a5}--\eqref{a8}
is that the Riemann--Liouville and Caputo fractional derivatives
are nonlocal operators. The left (right) differ-integration operator \eqref{a5}
and \eqref{a7} (\eqref{a6} and \eqref{a8}) depend on the values of the function
at left (right) of $x$, i.e., $a\leq u \leq x$ ($x\leq u \leq b$).

\begin{remark}
When $\alpha$ is an integer, the Riemann--Liouville fractional derivatives
\eqref{a5} and \eqref{a6} reduce to ordinary derivatives of order $\alpha$.
On the other hand, in that case, the Caputo derivatives \eqref{a7} and \eqref{a8}
differ from integer order ones by a polynomial of order $\alpha -1$
{\rm \cite{Kilbas,Diethelm}}.
\end{remark}

\begin{remark}
For $0<\alpha<1$, and because they are given
by a first order derivative of a fractional integral,
the Riemann--Liouville fractional derivatives \eqref{a5} and \eqref{a6}
can be applied to non-differentiable functions.
One can even take Riemann--Liouville derivatives
of nowhere differentiable functions,
like the Weierstrass function {\rm \cite{KolwankarGangal}}.
In fact, one of the main reasons for using fractional calculus
is the possibility to study non-differentiability
{\rm \cite{Jumarie,Xiong:Wang:FDA12}}.
\end{remark}

\begin{remark}
It is important to note that, while the Caputo derivatives
\eqref{a7} and \eqref{a8} of a constant are zero for all $\alpha>0$,
${_a^C D}^{\alpha}_x 1 = {_x^C D}^{\alpha}_b 1 = 0$,
the Riemann--Liouville derivatives \eqref{a5} and \eqref{a6}
are not zero for $\alpha \notin \mathbb{N}$:
\begin{equation}
\label{a10}
{_a D}^{\alpha}_x 1 = \frac{(x-a)^{-\alpha}}{\Gamma(1-\alpha)},
\quad {_x D}^{\alpha}_b 1 = \frac{(b-x)^{-\alpha}}{\Gamma(1-\alpha)}.
\end{equation}
For $\alpha \in \mathbb{N}$, the right hand sides of \eqref{a10} become equal
to zero due to the poles of the gamma function. Furthermore,
for the power function $(x-a)^{\beta}$ one has
\begin{equation*}
{_a D}^{\alpha}_x (x-a)^{\beta}
=\frac{\Gamma(\beta+1)}{\Gamma(\beta+1-\alpha)}(x-a)^{\beta-\alpha}.
\end{equation*}
\end{remark}

\begin{remark}
\label{r4}
Let $\alpha \in ]0,1[$.
If the Riemann--Liouville and the Caputo fractional derivatives of order $\alpha$ exist,
then they are connected with each other by the following relations:
$$
{_a^CD_x^\alpha}f(x)={_aD_x^\alpha}f(x)- \frac{f(a)}{\Gamma(1-\alpha)}(x-a)^{-\alpha}
$$
and
$$
{_x^CD_b^\alpha}f(x)={_xD_b^\alpha}f(x)-\frac{f(b)}{\Gamma(1-\alpha)}(b-x)^{-\alpha}.
$$
Therefore, if $f(a)=0$ ($f(b)=0$), then ${_a^C D}^{\alpha}_x f(x) = {_a D}^{\alpha}_x f(x)$
(${_x^C D}^{\alpha}_b f(x) = {_x D}^{\alpha}_b f(x)$) {\rm \cite{Kilbas,Diethelm}}.
\end{remark}

We make use of the following three properties
of the Riemann--Liouville and Caputo derivatives
(Theorems~\ref{thm:ml:01}, \ref{thm:3} and \ref{thm:ml:03})
in the proofs of our results.

\begin{theorem}[see, e.g., \cite{Kilbas,Diethelm}]
\label{thm:ml:01}
Let $\alpha>0$. For every $f\in L_1([a,b])$ one has
\begin{equation}
\label{a11}
{_a D}^{\alpha}_x {_a J}^{\alpha}_x f(x) = f(x),
\quad {_x D}^{\alpha}_b {_x J}^{\alpha}_b f(x) = f(x)
\end{equation}
and
\begin{equation}
\label{a12}
{_a^C D}^{\alpha}_x {_a J}^{\alpha}_x f(x) = f(x),
\quad {_x^C D}^{\alpha}_b {_x J}^{\alpha}_b f(x) = f(x),
\end{equation}
almost everywhere.
\end{theorem}

\begin{theorem}[Fundamental theorem of Caputo calculus --- see, e.g., \cite{Kilbas,Diethelm}]
\label{thm:3}
Let $0<\alpha<1$ and $f$ be a differentiable function in $[a,b]$.
The following two equalities hold:
\begin{equation}
\label{a13}
{_a J}^{\alpha}_b {_a^C D}^{\alpha}_x f(x) = f(b)-f(a)
\end{equation}
and
\begin{equation}
\label{a14}
{_b J}^{\alpha}_a {_x^C D}^{\alpha}_b f(x) = f(a)-f(b).
\end{equation}
\end{theorem}

\begin{theorem}[Integration by parts --- see, e.g., \cite{SKM}]
\label{thm:ml:03}
Let $0<\alpha<1$ and $f$ be a differentiable function in $[a,b]$.
For any function $g \in L_1([a,b])$ one has
\begin{equation}
\label{a15}
\int_{a}^{b} g(x) {_a^C D_x^{\alpha}} f(x)dx
= \int_a^b f(x) {_x D_b^{\alpha}} g(x)dx
+ \left[{_xD_b^{\alpha-1}} g(x)  f(x)\right]_a^b
\end{equation}
and
\begin{equation}
\label{a16}
\int_{a}^{b} g(x)  {_x^C D_b^\alpha} f(x)dx
=\int_a^b f(x) {_a D_x^\alpha} g(x) dx
+ \left[(-1)^{n}{_aD_x^{\alpha-1}} g(x) f(x)\right]_a^b.
\end{equation}
\end{theorem}

The proof of relations \eqref{a11}--\eqref{a14}
can be found in several textbooks --- see, e.g., \cite{Kilbas,Diethelm}.
Relation \eqref{a11} is a direct consequence of \eqref{a5}
and \eqref{a6}. Theorem~\ref{thm:3} is the generalization of the fundamental theorem
of calculus to the Caputo fractional calculus. It is important to mention that
\eqref{a13} and \eqref{a14} do not hold in the Riemann--Liouville approach.
Finally, we remark that the formulas of integration by parts \eqref{a15} and \eqref{a16}
relate Caputo left (right) derivatives to Riemann--Liouville right (left) derivatives.


\section{The Fractional DuBois--Reymond Lemma}
\label{sec:3}

The fundamental problem of the fractional calculus of variations consists
to find a function $y$ that maximizes or minimizes a given functional $J$
\cite{MyID:208}. We consider a functional with a Lagrangian depending on
the independent variable $x$, function $y$ and its left Caputo
fractional derivative of order $0<\alpha<1$,
\begin{equation}
\label{b1}
J[y]=\int_a^b L\left(x,y,{_a^C D}^{\alpha}_x y\right) dx
= \Gamma(\alpha) {_a J}^{\alpha}_b \left[ (b-x)^{1-\alpha}
L\left(x,y,{_a^C D}^{\alpha}_x y\right)\right],
\end{equation}
where $_aJ_b^\alpha$ is seen as the left Riemann--Liouville fractional integral at $x=b$,
subject to the boundary conditions $y(a)=y_a$ and $y(b)=y_b$, $y_a,y_b\in \mathbb{R}$.
For $\alpha=1$, $J$ is a functional of the classical calculus of variations \cite{MR0160139}.
In order to obtain a necessary condition for the extremum of \eqref{b1},
we use the following two lemmas.

\begin{lemma}
\label{lemma:1}
Let $f\in L_1([a,b])$ and $0<\alpha<1$. If there is a number
$\varepsilon \in ]a,b] $ such that
$$
|f(x)|\leq c(x-a)^{\beta}
$$
for all $x\in [a,\varepsilon]$ with $c>0$ and $\beta>-\alpha$, then
\begin{equation}
\label{b1a}
\lim_{x\rightarrow a^+} {_a J}^{\alpha}_x f(x)=0.
\end{equation}
\end{lemma}

\begin{proof}
From \eqref{a3} one has
\begin{equation*}
\left|\int_{a}^x \frac{f(u)}{(x-u)^{1-\alpha}}du\right|
\leq c\int_{a}^x \frac{(u-a)^{\beta}}{(x-u)^{1-\alpha}}du
=c\frac{\Gamma(\alpha)\Gamma(\beta+1)}{\Gamma(\alpha+\beta+1)}(x-a)^{\alpha+\beta}
\end{equation*}
for all $x\in [a,\varepsilon]$. Taking into account that $\alpha+\beta>0$,
it follows that
\begin{equation*}
\left|\lim_{x\rightarrow a^+}\int_{a}^x \frac{f(u)}{(x-u)^{1-\alpha}}du\right|
\leq c\frac{\Gamma(\alpha)\Gamma(\beta+1)}{\Gamma(\alpha+\beta+1)}
\lim_{x \rightarrow a^+} (x-a)^{\alpha+\beta}=0.
\end{equation*}
\end{proof}

\begin{remark}
If $0< \beta \leq 1$, then function $f$ of Lemma~\ref{lemma:1} is a H\"older function
of order $\beta$. If $\beta=0$, then $f$ is a simple bounded function.
\end{remark}

\begin{lemma}[The DuBois--Reymond fundamental lemma of the fractional calculus of variations]
\label{lemma2}
Let $g$ be a differentiable function in $[a,b]$ with $g(a)=g(b)=0$,
and let $f\in L_1([a,b])$ be such that there is a number $\varepsilon \in ]a,b]$
with $|f(x)|\leq c(x-a)^{\beta}$ for all $x\in [a,\varepsilon]$,
where $c>0$ and $\beta>-\alpha$ are constants. Then,
\begin{equation*}
{_a J}^{\alpha}_b \left(f(x) {_a^C D}^{\alpha}_x g(x)\right)=0
\Rightarrow f  \equiv K,
\end{equation*}
where $K$ is a constant.
\end{lemma}

\begin{proof}
For any constant $K$ we have
\begin{equation}
\label{b3}
\begin{split}
{_a J}^{\alpha}_b \left(\left(f(x)-K\right) {_a^C D}^{\alpha}_x g(x)\right)
&={_a J}^{\alpha}_b \left(f(x) {_a^C D}^{\alpha}_x g(x)\right)
-K {_a J}^{\alpha}_b {_a^C D}^{\alpha}_x g(x)\\
&= -K {_a J}^{\alpha}_b {_a^C D}^{\alpha}_x g(x)\\
&= - K\left(g(b)-g(a)\right)\\
&=0,
\end{split}
\end{equation}
where we used the fundamental theorem of Caputo calculus \eqref{a13}
and the hypothesis $g(a)=g(b)$. Let us choose
\begin{equation}
\label{b4}
g(x) := {_a J}^{\alpha}_x \left(f(x)-K\right).
\end{equation}
It can be seen that $g(x)$, defined by \eqref{b4}, is a differentiable function in $[a,b]$.
Furthermore, from \eqref{b1a} of Lemma~\ref{lemma:1}, we have $g(a)=0$ and, by choosing
\begin{equation*}
K=\frac{1}{{_a J}^{\alpha}_b 1} {_a J}^{\alpha}_b f
=\frac{\Gamma(\alpha+1)}{(b-a)^{\alpha}} {_a J}^{\alpha}_b f,
\end{equation*}
we also have $g(b)=0$. By inserting \eqref{b4} into \eqref{b3}, and using \eqref{a12}, we get
\begin{equation*}
{_a J}^{\alpha}_b \left(f(x)-K\right) {_a^C D}^{\alpha}_x g(x)
= {_a J}^{\alpha}_b \left(f(x)-K\right)^2
=\frac{1}{\Gamma(\alpha)}\int_a^b\frac{(f(u)-K)^2}{(b-u)^{1-\alpha}}du=0.
\end{equation*}
Since $b-x> 0$ for all $x\in [a,b[$, we conclude that
$f(x)=K$ for all $x\in [a,b]$.
\end{proof}

\begin{remark}
The fractional DuBois--Reymond lemma can be trivially formulated
for Riemann--Liouville derivatives instead of Caputo ones.
Indeed, ${_a^C D}^{\alpha}_x g(x)={_a D}^{\alpha}_x g(x)$ because $g(a)=0$,
as commented in Remark~\ref{r4}. Consequently, all results of the next section
can also be formulated for functionals depending on Riemann--Liouville derivatives.
\end{remark}


\section{Fractional Euler--Lagrange Equations}
\label{sec:4}

The next Theorem gives a necessary condition for a function
$y$ to be an extremizer of the fractional variational
functional defined by \eqref{b1}.

\begin{theorem}[The fractional Euler--Lagrange equation in integral form]
\label{theormDR}
Let $J$ be a functional of the form
\begin{equation*}
J[y]=\int_a^b L\left(x,y,{_a^C D}^{\alpha}_x y\right) dx
=\Gamma(\alpha) {_a J}^{\alpha}_b
\left[ (b-x)^{1-\alpha} L\left(x,y,{_a^C D}^{\alpha}_x y\right)\right],
\end{equation*}
defined in the class of functions $y \in C^1([a,b])$
satisfying given boundary conditions  $y(a)=y_a$ and $y(b)=y_b$,
and where $L\in C^1([a,b[ \times \mathbb{R}^2)$
is differentiable with respect to all of its arguments.
If $y$ is an extremizer of $J$, then $y$ satisfies
the following fractional Euler--Lagrange integral equation:
\begin{equation}
\label{b9}
{_x J}^{\alpha}_b \frac{\partial L\left(x,y,{_a^C D}^{\alpha}_x y\right)}{\partial y}
+\frac{\partial L(x,y,{_a^C D}^{\alpha}_x y)}{\partial ({_a^C D}^{\alpha}_x y)}
=\frac{K}{(b-x)^{1-\alpha}}
\end{equation}
for all $x\in [a,b[$, where $K$ is a constant.
\end{theorem}

\begin{proof}
Let $y^{\ast}$ give an extremum to \eqref{b1}.
We define a family of functions
\begin{equation}
\label{b10}
y(x)=y^{\ast}(x)+\epsilon \eta(x),
\end{equation}
where $\epsilon$ is a constant and $\eta \in C^1([a,b])$
is an arbitrary continuously differentiable function
satisfying the boundary conditions
$\eta(a)=\eta(b)=0$ (weak variations).
From \eqref{b10} and the boundary conditions $\eta(a)=\eta(b)=0$
and $y^{\ast}(a)=y_a$, $y^{\ast}(b)=y_b$,
it follows that function $y$ is admissible:
$y \in C^1([a,b])$ with $y(a)=y_a$ and $y(b)=y_b$.
Let the Lagrangian $L$ be $C^1([a,b[ \times \mathbb{R}^2)$.
Because $y^{\ast}$ is an extremizer of functional $J$,
the Gateaux derivative $\delta J[y^{\ast}]$
needs to be identically null. For the functional \eqref{b1},
\begin{equation}
\label{b12}
\begin{split}
\delta J[y^{\ast}]&=\lim_{\epsilon\rightarrow 0}
\frac{1}{\epsilon}\left( \int_a^b L\left(x,y,{_a^C D}^{\alpha}_x y\right)dx
-\int_a^b L(x,y^{\ast},{_a^C D}^{\alpha}_x y^{\ast})dx\right)\\
&=\int_a^b \left(\eta(x)\frac{\partial L(x,y^{\ast},
{_a^C D}^{\alpha}_x y^{\ast})}{\partial y^{\ast}}
+{_a^C D}^{\alpha}_x\eta(x)\frac{\partial L(x,y^{\ast},
{_a^C D}^{\alpha}_x y^{\ast})}{\partial ({_a^C D}^{\alpha}_x y^{\ast})} \right)dx\\
&=0.
\end{split}
\end{equation}
By using \eqref{a4} and relations \eqref{a11} and \eqref{a15}, we get
\begin{equation}
\label{b13}
\begin{split}
\int_a^b \eta(x)\frac{\partial L(x,y^{\ast},
{_a^C D}^{\alpha}_x y^{\ast})}{\partial y^{\ast}}dx
&=\int_a^b \eta(x) {_x D}^{\alpha}_b {_x J}^{\alpha}_b\frac{\partial
L(x,y^{\ast},{_a^C D}^{\alpha}_x y^{\ast})}{\partial y^{\ast}}dx \\
&=\int_a^b {_a^C D}^{\alpha}_x\eta(x) {_x J}^{\alpha}_b\frac{\partial
L\left(x,y^{\ast}, {_a^C D}^{\alpha}_x y^{\ast}\right)}{\partial y^{\ast}}dx.
\end{split}
\end{equation}
Inserting \eqref{b13} into \eqref{b12}, and using the definition \eqref{a3}
of left Riemann--Liouville fractional integral,
we obtain for the first variation the following expression:
\begin{equation*}
\begin{split}
\delta J[y^{\ast}] &= \int_a^b {_a^C D}^{\alpha}_x\eta(x) \left(
{_x J}^{\alpha}_b\frac{\partial L(x,y^{\ast},
{_a^C D}^{\alpha}_x y^{\ast})}{\partial y^{\ast}}
+\frac{\partial L(x,y^{\ast},{_a^C D}^{\alpha}_x y^{\ast})}{
\partial ({_a^C D}^{\alpha}_x y^{\ast})} \right) dx \\
&=\Gamma(\alpha) {_a J}^{\alpha}_b \left[{_a^C D}^{\alpha}_x\eta(x) \left(
{_x J}^{\alpha}_b\frac{\partial L(x,y^{\ast},
{_a^C D}^{\alpha}_x y^{\ast})}{\partial y^{\ast}}
+\frac{\partial L(x,y^{\ast},{_a^C D}^{\alpha}_x y^{\ast})}{\partial
\left({_a^C D}^{\alpha}_x y^{\ast}\right)} \right) (b-x)^{1-\alpha}\right]=0.
\end{split}
\end{equation*}
The fractional Euler--Lagrange equation \eqref{b9}
follows from Lemma~\ref{lemma2}. Note that
the hypothesis $|f(x)|\leq |c| (x-a)^{\beta}$ with $\beta>-\alpha$
is satisfied because $L\in C^1\left([a,b[ \times \mathbb{R}^2\right)$ and $y \in C^1([a,b])$.
Indeed, in our case $f(x)=\left({_x J}^{\alpha}_b\frac{\partial L(x,y^{\ast},
{_a^C D}^{\alpha}_x y^{\ast})}{\partial y^{\ast}}
+\frac{\partial L(x,y^{\ast},{_a^C D}^{\alpha}_x y^{\ast})}{\partial
\left({_a^C D}^{\alpha}_x y^{\ast}\right)} \right) (b-x)^{1-\alpha}$
tends to a number $c$ when $x \rightarrow a$. Therefore, $\beta=0>-\alpha$
and we are in conditions to apply Lemma~\ref{lemma2}.
\end{proof}

It is important to mention that in all previous approaches
to the fractional calculus of variations,
one eliminates the fractional derivative
of function $\eta$ appearing in the second integrand of \eqref{b12}
by performing an integration by parts \cite{MyID:208}.
However, from the integration by parts \eqref{a15},
this procedure gives an Euler--Lagrange fractional differential equation involving
Riemann--Liouville derivatives (see \cite{AMT} and references therein):
\begin{equation}
\label{b15}
\frac{\partial L\left(x,y,{_a^C D}^{\alpha}_x y\right)}{\partial y}
+{_x D}^{\alpha}_b\frac{\partial L\left(x,y,{_a^C D}^{\alpha}_x y\right)}{\partial
({_a^C D}^{\alpha}_x y)} = 0.
\end{equation}
In contrast, in the proof of Theorem~\ref{theormDR}
we take a different procedure by computing an integration by parts
on the fist integrand instead of the second one.
The new fractional DuBois--Reymond lemma takes then a prominent role.

One can obtain the Euler--Lagrange equation \eqref{b15} from our result,
by taking the right Riemann--Liouville derivative to both sides of \eqref{b9}:
\begin{equation}
\label{eq:myq2}
{_x D}^{\alpha}_b {_x J}^{\alpha}_b\frac{\partial L\left(x,y,{_a^C D}^{\alpha}_x y\right)}{\partial y}
+ {_x D}^{\alpha}_b\frac{\partial L\left(x,y,{_a^C D}^{\alpha}_x y\right)}{\partial({_a^C D}^{\alpha}_x y)}
=K {_x D}^{\alpha}_b\frac{1}{(b-x)^{1-\alpha}}.
\end{equation}
The Euler--Lagrange equation \eqref{b15} follows from \eqref{eq:myq2}
by Theorem~\ref{thm:ml:01} and the well-known equality
${_x D}^{\alpha}_b (b-x)^{\alpha - 1} = 0$
(see, e.g., Property~2.1 of \cite{Kilbas}).
A more important consequence from our optimality condition \eqref{b9}, however,
is that we can obtain a new fractional Euler--Lagrange differential equation
involving only Caputo derivatives.

\begin{theorem}[The Euler--Lagrange equation with only Caputo derivatives]
\label{cor:mr}
Consider the problem of extremizing \eqref{b1}
with a Lagrangian $L\in C^2([a,b]\times \mathbb{R}^2)$
subject to boundary conditions $y(a)=y_a$ and $y(b)=y_b$.
If $y \in C^1([a,b])$ is a solution to this problem, then $y$ satisfies
the fractional Euler--Lagrange differential equation
\begin{equation}
\label{b16}
\frac{\partial L\left(x,y,{_a^C D}^{\alpha}_x y\right)}{\partial y}
+ {_x^C D}^{\alpha}_b\frac{\partial L\left(x,y,{_a^C
D}^{\alpha}_x y\right)}{\partial ({_a^C D}^{\alpha}_x y)} = 0.
\end{equation}
\end{theorem}

\begin{proof}
The optimality condition \eqref{b16} is obtained
taking the right Caputo fractional derivative \eqref{a8}
to both sides of \eqref{b9}.
Taking the limit $x\rightarrow b$ on the left-hand side of \eqref{b9},
and having in mind that $L \in C^2([a,b]\times \mathbb{R}^2)$ and $y\in C^1([a,b])$,
we obtain the quantity $\frac{\partial L\left(b,y(b),{_a^C D}^{\alpha}_b y(b)\right)}{\partial
({_a^C D}^{\alpha}_b y)}$. On the other hand, the right-hand side of \eqref{b9} diverges when
$x\rightarrow b$ if $K\neq 0$. We conclude that, for $L\in C^2\left([a,b]\times \mathbb{R}^2\right)$,
one has $K=0$ and $\frac{\partial L\left(b,y(b),{_a^C D}^{\alpha}_b y(b)\right)}{\partial
({_a^C D}^{\alpha}_b y)}=0$ (this is a necessary condition for the Lagrangian $L$ and the solution $y$
to be nonsingular at $x=b$). Finally, we obtain \eqref{b16} by computing the Caputo derivative
of the right-hand side of \eqref{b9} and using Theorem~\ref{thm:ml:01}.
\end{proof}

From the well-known relations between the Riemann--Liouville and Caputo fractional derivatives
(see Remark~\ref{r4}), the Euler--Lagrange equation in terms of Caputo fractional derivative \eqref{b16}
can be derived from the Euler--Lagrange equation in terms of Riemann--Liouville
fractional derivative \eqref{b15} by assuming that
$\frac{\partial L\left(b,y(b),{_a^C D}^{\alpha}_b y(b)\right)}{\partial ({_a^C D}^{\alpha}_b y)}=0$.
Here, we do not assume a priori that
$\frac{\partial L\left(b,y(b),{_a^C D}^{\alpha}_b y(b)\right)}{\partial ({_a^C D}^{\alpha}_b y)}=0$,
showing it as a consequence of the new fractional Euler--Lagrange equation in integral form \eqref{b9}.
The Euler--Lagrange equation \eqref{b16}, involving only Caputo derivatives,
is valid for regular boundary conditions. Consequently, equation \eqref{b16}
should be better suited to applications in physics,
sciences and engineering than the fractional Euler--Lagrange equation \eqref{b15}
with mixed Riemann--Liouville and Caputo fractional derivatives \cite{HerrmannBook}.


\section{Conclusion}
\label{sec:5}

We generalized one of the most important lemmas of the calculus of variations,
the DuBois--Reymond fundamental lemma of variational calculus,
to functionals depending on fractional derivatives (Lemma~\ref{lemma2}).
The new lemma enabled us to prove a fractional Euler--Lagrange equation in integral form
containing fractional derivatives of only one type (Theorem~\ref{theormDR}).
Furthermore, we also showed that, when the Lagrangian is a $C^2$ function,
one can then obtain a fractional Euler--Lagrange differential equation
depending only on Caputo derivatives (Theorem~\ref{cor:mr}).
This is an important result because differential equations involving only Caputo derivatives are,
in general, better suited to applications in physics, sciences and engineering,
than a fractional differential equation involving
both Caputo and Riemann--Liouville derivatives \cite{HerrmannBook}.


\section*{Acknowledgments}

The authors are grateful to two referees for their valuable
comments and helpful suggestions.


\small



\end{document}